\documentclass[draft]{amsart}

\usepackage{times,fancyhdr}
\usepackage{amssymb}
\usepackage{array}
\usepackage{color}
\usepackage{ulem}

\theoremstyle{plain}       
\newtheorem{thm}{Theorem}[section]
\newtheorem{prop}{Proposition}[section]
\newtheorem{cor}{Corollary}[section]
\newtheorem{lemma}{Lemma}[section]
\theoremstyle{definition}

\newtheorem{example}{Example}[section]
\newtheorem{remark}{Remark}[section]

\numberwithin{equation}{section}

\renewcommand{\labelenumi}{{\upshape(\roman{enumi})}.}

\begin{document}

\title{Linear independence results for certain sums 
of reciprocals of Fibonacci and Lucas numbers}

\author{Daniel Duverney}
\address[D. Duverney]
{B{\^a}timent A1\\
110 rue du chevalier Fran{\c{c}}ais\\
59000 Lille\\France}
\email{daniel.duverney@orange.fr}
\author{Yuta Suzuki}
\address[Y. Suzuki]
{Graduate School of Mathematics\\Nagoya University\\
Chikusa-ku\\Nagoya 464-8602\\Japan}
\email{suzuyu1729@gmail.com}
\author{Yohei Tachiya}
\address[Y. Tachiya]
{Graduate School of Science and Technology\\
Hirosaki University\\Hirosaki 036-8561\\Japan}
\email{tachiya@hirosaki-u.ac.jp}
\maketitle
\begin{abstract}
The aim of this paper is to give linear independence results
for the values of certain series. 
As an application, we derive arithmetical properties of the 
sums of reciprocals of Fibonacci and Lucas numbers associated with certain
coprime sequences $\{n_\ell\}_{\ell\geq1}$. 
For example, the three numbers 
\[
1,\qquad \sum_{p\text{:prime}}^{}\frac{1}{F_{p^2}},\qquad \sum_{p\text{:prime}}^{}
\frac{1}{L_{p^2}}
\]
are linearly independent over $\mathbb{Q}(\sqrt{5})$, 
where $\{F_n\}$ and $\{L_n\}$ are the Fibonacci and Lucas numbers, respectively. 
\end{abstract}

\pagestyle{fancy}  
\fancyhead{}
\fancyhead[EC]{Daniel Duverney, Yuta Suzuki, and Yohei Tachiya}
\fancyhead[EL,OR]{\thepage}
\fancyhead[OC]{Linear independence results for sums of reciprocal 
of Fibonacci and Lucas numbers}
\fancyfoot{}
\renewcommand\headrulewidth{0.5pt} 

\section{Introduction and results}\label{sec:1}
Throughout this paper, let $\{n_\ell\}_{\ell\geq1}$ be an increasing sequence of 
positive odd integers satisfying the following two conditions:
\begin{enumerate}
\renewcommand{\labelenumi}{{\upshape($H_\arabic{enumi}$)}}
\item\label{H1}
Any two distinct integers $n_i$ and $n_j$ are coprime,
\item\label{H2}
$\sum_{\ell=1}^{\infty}\frac{1}{n_\ell}$ is convergent.
\end{enumerate}
\begin{example}\label{exa:prime_power}
It is well known that
the $\ell$th prime number $p_\ell$ is asymptotically equal to $\ell\log\ell$ as $\ell\to\infty$.
Hence, the sequence of $m$-th powers of odd primes $\{p_{\ell+1}^m\}_{\ell\geq1}$ 
satisfies the conditions $(H_1)$ and $(H_2)$ for $m\geq2$.  
\end{example}
\begin{example}\label{exa:super_prime}
The super-prime numbers (also known as prime-indexed primes) are
the subsequence of prime numbers
that occupy prime-numbered positions within the sequence of all prime numbers. 
Then the $\ell$th super-prime number $p_{p_{\ell}}$
is asymptotically equal to $p_\ell\log p_\ell\sim \ell(\log\ell)^2$ as $\ell\to\infty$,
and so the sequence of all super-prime numbers $\{p_{p_{\ell}}\}_{\ell\geq1}$ 
satisfies the conditions $(H_1)$ and $(H_2)$.
\end{example}
For any positive integer $t>1$, 
Erd{\H o}s \cite{Erd2} showed that the base-$t$ representation of the infinite series 
\begin{equation}\label{Erdos}
\sum_{\ell=1}^{\infty}\frac{1}{t^{n_\ell}-1}
\end{equation}
contains arbitrarily
long strings of $0$ without being identically zero from some point on, 
and consequently the number \eqref{Erdos} is irrational. 
The purpose of this paper is to improve Erd{\H o}s's method in \cite{Erd2} 
and give linear independence results for 
certain infinite series. 

Let $a_1(n)$ and $a_3(n)$ be the numbers of divisors $n_\ell$ of $n$ 
of the forms $4m+1$ and $4m+3$, respectively. 
For $j=1,2,3,4$, we define 
\begin{equation}\label{fs}
f_j(z)
:=\sum_{n=1}^{\infty}b_j(n)z^n,
\end{equation}
where 
\begin{equation}\label{bs}
b_j(n):=\left\{
\begin{array}{cl}
a_1(n)& \text{if}\,\,\,n\equiv j\pmod{4},\\
a_3(n)& \text{if}\,\,\,n\equiv j+2\pmod{4},\\
0&{\rm otherwise}.
\end{array}
\right.
\end{equation}
Note that the functions $f_j(z)$ $(j=1,2,3,4)$ 
converge for any complex number $z$ with $|z| < 1$, 
since $b_j(n)\leq n$ for $n\geq 1$. 
Our main result is the following.
\begin{thm}\label{thm:1}
Let $\alpha$ be an algebraic integer with $|\alpha|>1$ whose conjugates over 
$\mathbb{Q}$ other than 
itself and its complex conjugate lie in the open unit disk.   
Then  the five numbers
\begin{equation}\label{5numbers}
1,\quad f_1(\alpha^{-1}),\quad f_2(\alpha^{-1}),\quad f_3(\alpha^{-1}),
\quad f_4(\alpha^{-1})
\end{equation}
are linearly independent over the field $\mathbb{Q}(\alpha)$.  
\end{thm}
Let $\alpha$ be as in Theorem~\ref{thm:1}. Then the number $\alpha$ is called Pisot number or Pisot--Vijayaraghavan number, 
if $\alpha$ is a real positive number. 
Also, $\alpha$ is called complex Pisot number, if $\alpha$ is a non-real number. 
The Pisot numbers of degree one are exactly the rational integers greater than one. 

\vspace{0.2cm}

Theorem~\ref{thm:1} can be applied to obtain linear independence results for 
the values of certain Lambert series. 
For any complex number $z$ with $|z|<1$, we have the expressions 
\begin{align}
\sum_{\ell=1}^{\infty}\frac{z^{n_{\ell}}}{1\mp z^{n_{\ell}}}&=
\sum_{\ell=1}^{\infty}\sum_{k=1}^{\infty}(\pm1)^{k-1}z^{kn_{\ell}}=
\sum_{n=1}^{\infty}
\left(
\sum_{n_\ell\mid n}(\pm1)^{\frac{n}{n_\ell}-1}
\right)z^n\nonumber \\
&=f_1(z)+f_3(z)\pm(f_2(z)+f_4(z)),\nonumber \\
\sum_{\ell=1}^{\infty}\frac{z^{n_{\ell}}}{1\mp z^{2n_{\ell}}}&=f_1(z)\pm f_3(z).\label{exp13}
\end{align}
Hence, Theorem~\ref{thm:1} yields the following Corollary~\ref{cor:001}, 
which generalizes the irrationality result of 
Erd{\H o}s~\cite{Erd2}.     
\begin{cor}\label{cor:001}
Let $t$ be any rational integer with $|t|>1$. Then the four numbers 
\[
1,\qquad \sum_{\ell=1}^{\infty}
\frac{1}{t^{n_{\ell}}-1},\qquad 
\sum_{\ell=1}^{\infty}
\frac{1}{t^{n_{\ell}}+1},
\qquad
\sum_{\ell=1}^{\infty}
\frac{t^{n_\ell}}{t^{2n_{\ell}}-1}
\]
are linearly independent over $\mathbb{Q}$. 
\end{cor}
Let $\alpha$ be as in Theorem~\ref{thm:1} and $\beta:=\pm \alpha^{-1}$.
Define 
\begin{equation}\label{def:UV}
U_n=\frac{\alpha^n-\beta^n}{\alpha-\beta}\qquad\text{and}\qquad
V_n=\alpha^n+\beta^n\quad(n\ge1),
\end{equation}
which are the Lucas sequences of the first and second kind of parameters $\alpha$ and $\beta$. 
\begin{cor}\label{cor:1} 
Let $\{U_n\}_{n\geq1}$ and $\{V_n\}_{n\geq1}$ be the sequences defined by 
\eqref{def:UV}. 
Then the three numbers 
\begin{equation}\label{UV}
1,\qquad
\sum_{\ell=1}^{\infty}\frac{1}{U_{n_\ell}},\qquad 
\sum_{\ell=1}^{\infty}\frac{1}{V_{n_\ell}}
\end{equation}
are linearly independent over the field $\mathbb{Q}(\alpha)$. 
\end{cor}
Corollary~\ref{cor:1} follows immediately from Theorem~\ref{thm:1}. 
Indeed, recalling that all $n_\ell$ are odd and $\beta=\pm\alpha^{-1}$,  we have by \eqref{exp13} 
\begin{align*}
\frac{1}{\alpha-\beta}
\sum_{\ell=1}^{\infty}\frac{1}{U_{n_\ell}}&=
\sum_{\ell=1}^{\infty}
\frac{\alpha^{-n_\ell}}{1\mp \alpha^{-2n_\ell}}
=f_1(\alpha^{-1})\pm f_3(\alpha^{-1}),\\
\sum_{\ell=1}^{\infty}\frac{1}{V_{n_\ell}}&=
\sum_{\ell=1}^{\infty}
\frac{\alpha^{-n_\ell}}{1\pm \alpha^{-2n_\ell}}
=f_1(\alpha^{-1})\mp f_3(\alpha^{-1}).
\end{align*}
\begin{example}\label{exa:FL}
Putting $\alpha:=(1+\sqrt{5})/2$ and $\beta:=-\alpha^{-1}$ in (\ref{def:UV}),
we have $U_n=F_n$ and $V_n=L_n$, 
which are the classical Fibonacci and Lucas numbers
defined by $F_{n+2}=F_{n+1}+F_n$ $(n\geq0)$, $F_0=0$, $F_1=1$ and 
$L_{n+2}=L_{n+1}+L_n$ $(n\geq0)$, $L_0=2$, $L_1=1$, 
respectively. 
Hence, the three numbers 
\[
1,\qquad \sum_{\ell=1}^{\infty}\frac{1}{F_{n_\ell}},\qquad \sum_{\ell=1}^{\infty}
\frac{1}{L_{n_\ell}}
\]
are linearly independent over the field $\mathbb{Q}(\sqrt{5})$. 
From the view of Example~\ref{exa:prime_power}, 
the three numbers $1$, $\sum_{p}^{}1/F_{p^m}$, $\sum_{p}1/L_{p^m}$ are linearly independent over $\mathbb{Q}$ for any integer $m\geq2$, where the sums are taken over all prime numbers. 
\end{example}
Note that we are still unaware of the irrationalities of 
$\sum_{p}^{}1/F_{p}$ and $\sum_{p}1/L_{p}$. 
\begin{remark}\label{rmk1}
In 1989, R.~Andr\'e-Jeannin~\cite{And} proved the irrationality of 
the fundamental sum
$f:=\sum_{n=1}^{\infty}1/F_n$; see also \cite{Bun2,Duv,Tac}. 
More generally, P. Bundschuh and K. V{\"a}{\"a}n{\"a}nen~\cite{BV} 
obtained $f\notin \mathbb{Q}(\sqrt{5})$ as well as 
an irrationality measure. 
Much is known about 
the quantitative result of $f$; 
see, e.g., \cite{MP0,MP,Pre} 
on this direction. 
On the other hand, we know very little about linear independence results; for example, 
of the three numbers $1$, $f$, $\sum_{n=1}^{\infty}1/L_n$ over $\mathbb{Q}(\sqrt{5})$. 
For details around the series involving Fibonacci and Lucas numbers, refer to the survey \cite{DS}. 
\end{remark}
Our paper is organized as follows. Let 
\begin{equation}\label{E_decomp}
\{n_\ell\mid\ell=1,2,\dots\}=\mathcal{E}_1\cup\mathcal{E}_3,
\end{equation}
where the sets $\mathcal{E}_1:=\{u_n\mid u_1<u_2<\cdots\}$ and $\mathcal{E}_3:=\{v_n\mid v_1<v_2<\cdots\}$
consist of all positive integers in $\{n_\ell\}_{\ell\geq1}$ congruent to $1$ and $3$ modulo $4$, respectively.
In Section~\ref{sec:2}, we prepare some lemmas 
in accordance with the situation 
whether $\mathcal{E}_1$ and $\mathcal{E}_3$ are both infinite sets or not. 
Section~\ref{Sec3} is devoted to the proof of Theorem~\ref{thm:1}. 
The methods used in our proof are inspired by the original approach of Erd\H{o}s \cite{Erd2}, but we need a different technique in constructing the system of simultaneous congruences.
\section{Some properties of the coefficients $b_j(n)$}\label{sec:2} 
C. L.~Siegel~\cite{Sie} has shown that the smallest Pisot number is 
$\theta_0\approx 1.3247$, which is the unique real root of the polynomial $x^3-x-1$. 
The similar result for complex Pisot number was obtained by 
C. Chamfy \cite{Cha} who proved that the smallest modulus of a complex Pisot number 
is $\sqrt{\theta_0}\approx 1.1509$ (cf. \cite{Gar}). 
Hence, we have $|\alpha|\geq \sqrt{\theta_0}$ 
for given number $\alpha$ in Theorem~\ref{thm:1}, 
so that in particular $|\alpha|^5>2$.  

Moreover, if we remove a finite number of terms from the sequence $\{n_{\ell}\}_{\ell\geq1}$, 
then the new sequence $\{n_\ell'\}_{\ell\geq1}$ also 
satisfies the conditions $(H_1)$ and $(H_2)$. Hence, for proving 
Theorem~\ref{thm:1}, we may assume without loss of generality 
that  
\begin{equation}\label{n_l}
n_\ell>64 \quad (\ell\geq1).
\end{equation}
We first construct arbitrarily long sequences of consecutive integers $n$
on which all coefficients $b_j(n)$ $(j=1,2,3,4)$ take some prescribed values 
\textit{exactly} (see Lemmas~\ref{lem:1} and \ref{lem:2}). 
After that, we give upper bound results of the coefficients $b_j(n)$ 
for the integers $n$ surrounding such long sequences 
(see Lemmas~\ref{lem:3} and \ref{lem:4}). 
This construction plays an important role in producing  long gaps 
in the linear form of the infinite series \eqref{fs} over $\mathbb{Q}$. 
Let $k$ be a sufficiently large positive integer, 
which is used for the length of our sequences of consecutive integers.

In what follows, we distinguish two cases according to whether 
the sets $\mathcal{E}_1$ and $\mathcal{E}_3$ are both infinite or not.

\subsection{The case where $\mathcal{E}_1$ and $\mathcal{E}_3$ are both infinite}
\label{subsec2.1}
Let $\{x_m\}_{m\geq0}$ and $\{y_m\}_{m\geq0}$ be increasing sequences 
of nonnegative integers with $x_0=y_0=0$. 
We consider  the following system of $8k-3$ simultaneous congruences 
\begin{equation}\label{mod4}
X\equiv0\pmod{4},
\end{equation}
and 
\begin{equation}\label{cong1}
X+m\equiv0\quad \pmod{\prod_{x_{m-1}<n\le x_m}u_n\cdot
\prod_{y_{m-1}<n\le y_m}v_n}.
\end{equation}
We can group eight by eight the $8k-4$ congruences in \eqref{cong1} 
by defining the integers $q$ and $r$ such that 
\begin{equation}\label{8qr}
m=8q+r,
\end{equation}
where $0\le q\le k-2$ and $1\le r\le 8$ when $1\le m\le 8k-8$, 
and $q=k-1$ and $1\le r\le 4$ when $8k-7\le m\le 8k-4$. 
Hence, we have $q=0$ and $r=1,2,\dots,8$ successively for the first 
eight congruence, then $q=1$ and $r=1,2,\dots,8$ successively for the following 
eight congruences, and so on. We prove  
\begin{lemma}\label{lem:0}
There exist increasing sequences of nonnegative integers $\{x_m\}_{m\geq0}$ 
and $\{y_m\}_{m\geq0}$ with $x_0=y_0=0$ such that any solution $X$ 
of the system of simultaneous congruences \eqref{mod4} and \eqref{cong1} 
fulfills the following two conditions.
\begin{enumerate}
\renewcommand{\labelenumi}{{\upshape($C_\arabic{enumi}$)}}
\item 
For $1\le m\le 8k-8$, $X+m=X+8q+r$ is divisible by 
exactly $2^q$ integers $u_n$ with $n\le x_m$ 
and exactly $2^q$ integers $v_n$ with $n\le y_m$. 
\item 
For $8k-7\le m\le 8k-4$, $X+m=8(k-1)+r$ is divisible 
by exactly $k^r2^{k-1}$ integers $u_n$ with $n\le x_m$ 
and exactly $2^{k-1}$ integers $v_n$ with $n\le y_m$.
\end{enumerate}
\end{lemma}
\begin{proof}
We only give the details for $x_m$, since 
the same applies for $y_m$. 
For the first eight congruences, where $q=0$, 
we can take $x_1=1$, $x_2=x_1+1=2,\dots,x_8=x_7+1=8$, 
since $u_n>64$ for every 
$n\geq1$ by \eqref{n_l}.
Similarly, for the next eight, where $q=1$, we can take $x_9=x_8+2=10$, $x_{10}=x_9+2=12,\dots,x_{16}=x_{15}+2=24$. 
We can go on this way as long as $X+m$ is not a multiple of some $u_n$ 
which has already been used in the previous congruences, which is the case 
when $m\le 64$ (that is $q\le 7$). 
To be precise, the values of $x_m$ for $0\le q\le 7$ are given by $x_0=0$ and 
\[
x_m-x_{m-1}=2^q\qquad (1\le m\le 64).
\]
When $64<m\le 8k-8$ (that is $8\le q\le k-2$), we can not take all the following 
$u_n$ by this pattern, since $X+m$ can be divisible by some $u_n$ 
used in the previous congruences. 
In this case, we have to use the formula
\begin{equation}\label{x1}
x_m-x_{m-1}=2^q-s_m\qquad (64<m\le 8k-8),
\end{equation}
where $s_m$ denotes the number of $u_n$ with $1\le n\le x_{m-1}$ such that 
$X+m\equiv 0$ (mod $u_n$). 
We have to check that this formula defines an increasing sequence, that is 
that $s_m<2^q$. 
For this, we observe that, by definition \eqref{x1}, 
\begin{equation}\label{x2}
x_m\le 8(2^q+2^{q-1}+\cdots+2+1)<2^{q+4},
\end{equation} 
whence it follows from \eqref{n_l} and \eqref{x2} that $s_m\le x_{m-64}<2^{(q-8)+4}<2^q$. 
Therefore, $x_m$ defined by \eqref{x1} is increasing. 
Finally, when $8k-8<m\le 8k-4$ (that is when $q=k-1$ and $1\le r\le 4$), 
we use the formula
\begin{equation}\label{x3}
x_m-x_{m-1}=k^r2^{k-1}-s_m\qquad (8k-8<m\le 8k-4).
\end{equation}
For the integers $y_m$, we will have similarly 
\begin{align}
y_m-y_{m-1}&=2^q\qquad (1\le m\le 64),\nonumber\\
y_m-y_{m-1}&=2^q-t_m\qquad (64<m\le 8k-4),\label{y1}
\end{align}
where $t_m$ denotes the number of $v_n$ with $1\le n\le y_{m-1}$ such that 
$X+m\equiv0$ (mod $v_n$), and the proof of Lemma~\ref{lem:0} is 
completed.  
\end{proof}
By definitions \eqref{x3} and \eqref{y1}, we obtain   
\begin{equation}\label{8k-4}
x_{8k-4}>k^42^{k-1},\qquad y_{8k-4}>2^{k-1},
\end{equation}
since $s_m\le x_{m-64}<x_{m-1}$ and $t_m\le y_{m-64}<y_{m-1}$. 
Since the odd integers $u_n$ and $v_n$ are relatively prime, by the Chinese remainder theorem,
there exists a unique integer solution $\eta_k$ with $0\le \eta_k< A_k$
of the simultaneous congruences \eqref{mod4} and \eqref{cong1}, 
where 
\begin{equation}\label{def:Ak}
A_k:=4\prod_{n=1}^{x_{8k-4}}u_{n}\prod_{n=1}^{y_{8k-4}}v_n.
\end{equation}
Let $\mu_k$ be a positive integer defined by 
$n_{\mu_k}:=\min\{
u_{x_{8k-4}},v_{y_{8k-4}}
\}$ and 
\[
\delta_k:=
\exp\left(-16k\sum_{\ell>{{\mu_k}}}^{}\frac{1}{n_\ell}\right).
\]
Note that $\delta_k(<1)$ is well-defined by the condition $(H_2)$. 
Now we choose the least positive integer $\nu_k$ satisfying  
\begin{equation}\label{def:nuk}
\nu_k\ge\frac{128kA_k}{\delta_k}\quad\text{and}\quad 
\sum_{\ell>\nu_k}^{}\frac{1}{n_\ell}<\frac{\delta_k}{32k},
\end{equation}
which is possible, since $\sum_{\ell=1}^{\infty}1/n_\ell<\infty$ by 
the condition $(H_2)$. 
We divide the set \eqref{E_decomp} into the three sets as follows;
\[
\{n_\ell\mid \ell=1,2,\dots\}=\bigcup_{i=1}^{3}\mathcal{F}_i(k),
\]
where the sets $\mathcal{F}_i(k)$ $(i=1,2,3)$ are defined by 
\begin{align}
\mathcal{F}_1(k)&:=\left\{ u_{n},v_{n}\mid u_n\le u_{x_{8k-4}}, v_n\le v_{y_{8k-4}}\right\},\label{F_1(k)}\\
\mathcal{F}_2(k)&:=\left\{ u_{n},v_{n}\mid u_{x_{8k-4}}<u_{n}\leq n_{\nu _{k}}, v_{y_{8k-4}}<v_{n}\leq n_{\nu _{k}}\right\},\label{F_2(k)}\\
\mathcal{F}_3(k)&:=\left\{ u_{n},v_{n}\mid n_{\nu _{k}}<u_n, n_{\nu _{k}}\le v_n\right\}.\label{F_3(k)}
\end{align}
Clearly, the sets $\mathcal{F}_1(k)$ and $\mathcal{F}_3(k)$ are nonempty. 
Moreover, so is $\mathcal{F}_2(k)$, since we have by \eqref{def:Ak} and \eqref{def:nuk} 
\[
u_{x_{8k-4}},v_{y_{8k-4}} <A_{k}<\nu_k\leq n_{\nu _{k}},
\]
so that $n_{\nu_k}\in\mathcal{F}_2(k)$. Define 
\begin{equation}\label{def:Bk}
B_{k}:=\prod_{n_{\ell }\in \mathcal{F}_2(k)}n_{\ell }=\frac{4}{A_{k}}
\prod_{\ell =1}^{\nu _{k}}n_{\ell }.
\end{equation}
Then by definition \eqref{def:Bk} and the first property in (\ref{def:nuk}) we have 
\begin{equation}\label{AB_bound}
A_kB_k=4\prod_{\ell=1}^{\nu_k}n_\ell
\ge3^{\nu_k}\ge\exp(\nu_k)
\ge\exp\left(\frac{128kA_k}{\delta_k}\right),
\end{equation}
which implies particularly that 
\begin{equation}\label{B_bound}
B_k\geq \frac{1}{A_k}\exp\left(128kA_k\right)\geq 128k.
\end{equation} 
Let $\mathcal{G}(k)$ be the set of the $B_k$ positive integers 
\begin{equation}\label{defG_k}
\mathcal{G}(k):=\{\gamma_{i}:=A_ki+\eta_k\mid i=1,2,\dots,B_k\}. 
\end{equation}
\begin{lemma}\label{lem:1}
Let $b_j(n)$ $(j=1,2,3,4)$ be the integer sequences defined in \eqref{bs}. 
For every $m=1,2,\dots,8k-4$, let $q$ and $r$ be defined by \eqref{8qr}. 
Then the set 
$\mathcal{G}(k)$ contains at least $\delta_kB_k/2$
integers $\gamma$ such that 
\begin{align}\label{b_values1}
b_j(\gamma+m)&=
\left\{
\begin{array}{rl}
0&\quad \text{\rm if} \,\,\,r\not\equiv j\pmod{2},\\
k^r2^{k-1}&\quad \text{\rm if} \,\,\,q=k-1,\,\,\,r\equiv j\pmod{4},\\
2^{q}&\quad {\rm otherwise}\\
\end{array}
\right.
\end{align}
for any $m=8q+r=1,2,\dots,8k-4$ and for any $j=1,2,3,4$.
\end{lemma}
\begin{proof}
For any $\gamma\in\mathcal{G}(k)$, 
the conditions ($C_1$) and ($C_2$) imply that
each integer $\gamma+m$ is divisible by
exactly $2^{q}$ ($k^r2^{k-1}$, if $q=k-1$) integers $u_n$ 
with $n\leq x_m$
and
exactly $2^{q}$ integers $v_n$ with $n\leq y_m$.
Hence, the properties \eqref{b_values1} are satisfied 
if the integer $\gamma+m$ 
is not divisible by any $u_n$ with $n>x_m$ nor any  
$v_n$ with $n>y_m$. 
The proof of Lemma~\ref{lem:1} is proceeded in three steps.

\vspace{0.2cm}

\noindent
{\it First step}. We consider first the integers $u_n$ with $x_m<n\le x_{8k-4}$ 
and the integers $v_n$ with $y_m<n\le y_{8k-4}$, 
which are the elements of $\mathcal{F}_1(k)$ defined in \eqref{F_1(k)}.  
We prove that, for any 
$\gamma\in\mathcal{G}(k)$ and for fixed $m$ with $1\le m\le 8k-4$, we have  
\begin{equation}\label{notdiv}
\left\{
\begin{array}{rll}
u_n\nmid\gamma+m&\quad \text{for}\,\,\,x_m<n\le x_{8k-4},\\
v_n\nmid\gamma+m&\quad \text{for}\,\,\,y_m<n\le y_{8k-4}.
\end{array}
\right.
\end{equation}
Indeed, otherwise there exist for example $\gamma_0\in\mathcal{G}(k)$ and 
the integers 
$m_0,n_0$ with $1\le m_0<8k-4$ and 
$x_{m_0}<n_0\le x_{8k-4}$ such that $u_{n_0}\mid \gamma_0+m_0$. 
Take $m_1>m_0$ such that 
\begin{equation}\label{n_0}
x_{m_1-1}<n_0\le x_{m_1}.
\end{equation}
Then $u_{n_0}\mid\gamma_0+m_1$ by \eqref{cong1}.
Hence, we have $u_{n_0}\mid m_1-m_0$, 
which implies $u_{n_0}<m_1$. 
On the other hand,  by \eqref{n_0} we have 
$u_{n_0}>n_0> x_{m_1-1}\ge m_1-1$. 
Therefore, we obtain $m_1-1<u_{n_0}<m_1$. 
This is impossible, since $u_{n_0}$ is an integer. 
Thus, \eqref{notdiv} is proved. 

\vspace{0.2cm}

\noindent
{\it Second step}.
We consider the integers $u_n$ and $v_n$ such that 
$u_{x_{8k-4}}<u_n\le n_{\nu_k}$ and $v_{y_{8k-4}}<v_n\le n_{\nu_k}$, 
which are the elements of $\mathcal{F}_2(k)$ defined in \eqref{F_2(k)}.  
We estimate the number of elements of the set 
\[
\mathcal{S}(k)
:=
\left\{
\gamma\in \mathcal{G}(k)\,\left|\,
\substack{
\displaystyle\text{The integers $\gamma+1,\gamma+2,\dots,\gamma+8k-4$}\\
\displaystyle\text{
are not divisible by any $u_n$, $v_n$ in $\mathcal{F}_2(k)$
}}
\right\}
\right.
\]
by using the inclusion-exclusion principle. 
For this, let $\mathcal{D}:=\{d_1,d_2,\dots,d_s\}$ 
be a nonempty subset of $\mathcal{F}_2(k)$ 
and $\mathcal{H}_{\mathcal{D}}$ be the 
set of $\gamma\in\mathcal{G}(k)$ 
such that the set of the consecutive integers  
$\{\gamma+1,\gamma+2,\dots,\gamma+8k-4\}$ contains 
multiples of all $d\in\mathcal{D}$. 
Let 
\[
{T}:=\{(t_1,t_2,\dots,t_s)\in\mathbb{Z}^s\mid 1\leq t_j\leq 8k-4,\,\,\,j=1,2,\dots,s\}.
\]
For $\mathbf{t}=(t_1,t_2,\dots,t_s)\in T$, we consider the set 
\[
{\mathcal H}_{\mathcal D}^{(\mathbf{t})}:=
\{\gamma\in\mathcal{H}_{\mathcal{D}}\mid 
\gamma+t_j\equiv0\,\,\,({\rm mod} \,\,d_j),\,\,\,j=1,2,\dots,s\}.
\]
Then we have 
\begin{equation}\label{H}
{\mathcal H}_{\mathcal D}=\displaystyle\bigcup_{\mathbf{t}\in T}
 {\mathcal H}_{\mathcal D}^{(\mathbf{t})},
\end{equation}
\begin{equation}\label{disjoint}
{\mathcal H}_{\mathcal D}^{(\mathbf{t_1})}\cap{\mathcal H}_{\mathcal D}^{(\mathbf{t_2})}=\emptyset \,\,\,\text{for any} \,\,\,
\mathbf{t_1,t_2}\in T \,\,\,\text{with}\,\,\, \mathbf{t_1\neq t_2}.
\end{equation}
It is clear that \eqref{H} follows from definitions of 
${\mathcal H}_{\mathcal D}$ 
and 
${\mathcal H}_{\mathcal D}^{(\mathbf{t})}$.  
To see \eqref{disjoint}, we suppose to the 
contrary that there exists a  
$\gamma\in{\mathcal H}_{\mathcal D}^{(\mathbf{t_1})}\cap
{\mathcal H}_{\mathcal D}^{(\mathbf{t_2})}$ 
for some $\mathbf{t_1,t_2}\in T$ with $\mathbf{t_1\neq t_2}$. 
Let $\mathbf{t_i}:=(t_{i,1},t_{i,2},\dots,t_{i,s})$ $(i=1,2)$. 
Since $\mathbf{t_1\neq t_2}$, there exists 
an integer $j$ such that $t_{1,j}\neq t_{2,j}$ and 
\[
\gamma+t_{1,j}\equiv0, \quad \gamma+t_{2,j}\equiv0\pmod{d_j}. 
\]
Thus, the integer $t_{1,j}-t_{2,j}$ is divisible by $d_j$. 
However by \eqref{8k-4} 
\[
0<|t_{1,j}-t_{2,j}|\le 
8k-4<2^{k-1}<\min\{x_{8k-4},y_{8k-4}\}\leq \min\{u_{x_{8k-4}},v_{y_{8k-4}}\}<d_j.\] 
This is a contradiction. Hence, by \eqref{H} and \eqref{disjoint} we obtain 
\begin{equation}\label{orderH}
|\mathcal{H}_{\mathcal{D}}|=\sum_{\mathbf{t}\in T}^{}|\mathcal{H}_{\mathcal D}^{(\mathbf{t})}|.
\end{equation}
Moreover, since the integers $A_k$ and $d_j$ are coprime, 
we find by the Chinese Remainder Theorem that 
for any given $(t_1,t_2,\dots,t_s)\in T$, there exists an integer~$i_0$ satisfying the $s$ congruences 
\[
\gamma_{i_0}+t_j:=A_ki_0+\eta_k+t_j\equiv 0\pmod{d_j},\quad j=1,2,\dots,s,
\]
where $i_0$ is uniquely determined modulo $d_1d_2\cdots d_s$. 
Thus, for any $\mathbf{t}\in T$ the set ${\mathcal H}_{\mathcal D}^{(\mathbf{t})}$ 
can be rewritten as  
\[
{\mathcal H}_{\mathcal D}^{(\mathbf{t})}=\{\gamma_i\in\mathcal{H}_{\mathcal{D}}\mid 
i\equiv i_0 \,\,(\text{\rm mod}\,\,{d_1d_2\cdots d_s}), \,\,
1\leq i\leq B_k\},
\]
and hence, noting that the integer $B_k$ is divisible by $d_1d_2\cdots d_s$, we obtain   
\begin{equation}\label{H_D}
|{\mathcal H}_{\mathcal D}^{(\mathbf{t})}|=\frac{B_k}{d_1d_2\cdots d_s}.
\end{equation}
Combining \eqref{orderH} and \eqref{H_D} gives 
\[
|\mathcal{H}_{\mathcal{D}}|=\sum_{\mathbf{t}\in T}^{}|\mathcal{H}_{\mathcal D}^{(\mathbf{t})}|=
\frac{B_k}{d_1d_2\cdots d_s}|T|=
(8k-4)^{|\mathcal{D}|}\frac{B_k}{\prod_{d\in \mathcal{D}}d}.
\] 
Therefore, by the inclusion-exclusion principle, we have 
\begin{equation}\label{G_shifted}
|\mathcal{S}(k)|=B_k+\sum_{\mathcal{D}
\subset\mathcal{F}_2(k)}(-1)^{|\mathcal{D}|}|\mathcal{H}_{\mathcal{D}}|
=
B_k\prod_{d\in\mathcal{F}_2(k)}\left(1-\frac{8k-4}{d}\right)
\geq \delta_kB_k,
\end{equation}
where we used 
\begin{align*}\label{log:d}
\log \prod_{d\in\mathcal{F}_2(k)}
\left(1-\frac{8k-4}{d}\right)
&=
\sum_{d\in\mathcal{F}_2(k)}
\log\left(1-\frac{8k-4}{d}\right)\\
{}&>
-2\sum_{d\in\mathcal{F}_2(k)}^{}\frac{8k-4}{d}
\ge
-16k\sum_{\ell>{{\mu_k}}}^{}\frac{1}{n_\ell}
=\log\delta_k,\nonumber
\end{align*}
since $\log(1-x)> -2x$ holds for  sufficiently small $x>0$.  


\vspace{0.2cm}

\noindent
{\it Third step}.
We consider here the integers $u_n>n_{\nu_k}$ and $v_n>n_{\nu_k}$, 
which are the elements of $\mathcal{F}_3(k)$ defined in \eqref{F_3(k)}.  
For a fixed integer $t$ $(1\leq t\leq 8k-4)$,  
the number of $\gamma=A_ki+\eta_k\in\mathcal{G}(k)$ satisfying 
\[
n_\ell\mid\gamma+t=A_ki+\eta_k+t,\qquad n_\ell >n_{\nu_k}, 
\]
is at most $\lfloor B_k/{n_\ell}\rfloor+1$. 
Hence, the number of 
integers $\gamma\in \mathcal{G}(k)$ such that at least one of 
the integers $\gamma+1,\gamma+2,\dots,\gamma+8k-4$ is divisible by 
some $n_\ell>n_{\nu_k}$ is at most 
\begin{equation}\label{eq:9}
(8k-4)\cdot {\sum}'\left(\left\lfloor \frac{B_k}{n_\ell}\right\rfloor+1\right),
\end{equation}
where the sum is taken over all integers $n_\ell$ with  
\begin{equation}\label{eq:423}
n_{\nu_k}<n_\ell\le A_k B_k+\eta_k+8k-4\le 2A_kB_k.
\end{equation}
Let $\pi(x)$ denote the number of primes $p\leq x$. 
Clearly, the number of the integers 
$n_\ell$ satisfying (\ref{eq:423}) is less than   
\[
\pi(2A_kB_k)
< 2\frac{2A_kB_k}{\log (2A_kB_k)}
\le \frac{\delta_k}{32k}B_k,
\]
where we used (\ref{AB_bound}) and the Prime Number Theorem. 
Thus, the sum in (\ref{eq:9}) is taken over 
at most ${\delta_k}B_k/(32k)$ integers $n_\ell$, and by the second property in (\ref{def:nuk})  
\begin{equation}\label{eq:496}
\displaystyle{\sum}'\left(\left\lfloor \frac{B_k}{n_\ell}\right\rfloor+1\right)
\le
\mathop{\displaystyle{\sum}'}1+B_k\displaystyle\sum_{\ell>\nu_k}\frac{1}{n_\ell}
\le
\frac{\delta_k}{16k}B_k.
\end{equation}
Hence, by (\ref{eq:9}) and (\ref{eq:496}) the number of 
integers $\gamma\in \mathcal{G}(k)$ such that at least one of 
the integers $\gamma+1,\gamma+2,\dots,\gamma+8k-4$ is divisible by 
some $n_\ell>n_{\nu_k}$ is at most 
\begin{equation}\label{eq:858}
(8k-4)\cdot{\sum}'\left(\left\lfloor\frac{B_k}{n_\ell}\right\rfloor+1\right)
\le\frac{1}{2}\delta_k B_k.
\end{equation}

Therefore, combining \eqref{notdiv}, \eqref{G_shifted}, and \eqref{eq:858}, we find that 
the number of $\gamma\in \mathcal{G}(k)$ 
such that each integer 
$\gamma+m$ 
is not divisible by any $u_n$ with $n>x_m$ 
and $v_n$ with $n>y_m$ 
is at least $\delta_kB_k/2$ integers.
The proof of  Lemma~\ref{lem:1} is completed.
\end{proof}
\begin{lemma}\label{lem:3} 
Let $\xi>1$ be an arbitrary constant and assume that $k$ is sufficiently large 
depending on $\xi$.
Then there exist at least $(1-\delta_k/4)B_k$ integers $\gamma\in \mathcal{G}(k)$ such that 
\begin{equation}\label{eq:103}
b_j(\gamma+4k+i)< \xi^{|i|}
\end{equation}
for any integer $j=1,2,3,4$ and for any integer $i$ with 
$-\gamma-4k<i\leq -4k$ or $i\geq4k-3$. 
\end{lemma}
\begin{proof} 
By definition \eqref{bs} we have 
\begin{equation}\label{eq:15}
b_j(\gamma+4k+i)
\le\sum_{n_\ell\mid \gamma+4k+i}^{}1=\lambda_1( \gamma+4k+i)+\lambda_2( \gamma+4k+i)
\end{equation}
for any integer $i$ and for any $j=1,2,3,4$, where
\[
\lambda_1(\gamma+4k+i)
:=
\sum_{
\substack{n_\ell\mid \gamma+4k+i\\
n_\ell\in\mathcal{F}_1(k)}
}1,\qquad
\lambda_2(\gamma+4k+i)
:=
\sum_{
\substack{n_\ell\mid \gamma+4k+i\\
n_\ell\notin\mathcal{F}_1(k)}
}1,
\]
and $\mathcal{F}_1(k)$ is defined by \eqref{F_1(k)}. 
Let $\xi>1$. 
We first show  
\begin{equation}\label{eq:11}
\lambda_1(\gamma+4k+i)<\xi^{|i|}/2
\end{equation}
for any $\gamma\in \mathcal{G}(k)$ and for any integer $i$ with  $-\gamma-4k<i\leq -4k$ or
$i\geq 4k-3$.  
Assume that $n_\ell\mid \gamma+4k+i$ with $n_\ell\in\mathcal{F}_1(k)$. 
By the congruences \eqref{cong1},
we have $n_\ell\mid\gamma+h$ for some integer $h$ with $1\le h \le 8k-4$,  
from which it follows $n_\ell\mid 4k+i-h$. 
Then we have $4k+i-h\neq0$, so that  
\[
n_\ell\leq |4k+i-h|\leq |i|+12k\leq5|i|\leq \xi^{|i|}/2,
\]
since $|i|\geq3k$ is sufficiently large.
Thus, we obtain (\ref{eq:11}). 
Moreover, if  $i> A_kB_k/2$, then  
\begin{equation}\label{eq:097}
\lambda_2(\gamma+4k+i)<\xi^i/2
\end{equation}
holds for any $\gamma\in\mathcal{G}(k)$, since by \eqref{B_bound} and \eqref{defG_k}
\begin{align*}
n_\ell&\le\gamma+4k+i\le A_kB_k+\eta_k+4k+i\le A_kB_k+A_k+4k+i\\
{}&\leq A_kB_k+4kA_k+i\leq 2A_kB_k+i<5i< \xi^i/2. 
\end{align*}
Now we estimate the number of $\gamma\in\mathcal{G}(k)$ such that
\begin{equation}\label{eq:12}
\lambda_2(\gamma+4k+i)\geq \xi^{|i|}/2
\end{equation} 
holds for some $i$ with $-\gamma-4k<i\leq -4k$ or $4k-3\leq  i \leq A_kB_k/2$. 
Fix an integer $i$ and let $N_i$ denote the number of $\gamma\in\mathcal{G}(k)$ satisfying (\ref{eq:12}). 
Note that   
\[
1\leq \gamma+4k+i\leq A_kB_k+\eta_k+4k+A_kB_k/2<2A_kB_k
\]
for any $\gamma\in \mathcal{G}(k)$. 
Then, by the same argument used at the third step in proof of Lemma~\ref{lem:1}, 
we obtain
\begin{align*}
\frac{\xi^{|i|}}{2}N_{i}
&\le\sum_{\gamma\in\mathcal{G}(k)}\lambda_2(\gamma+4k+i)
=\sum_{\gamma\in\mathcal{G}(k)}\sum_{\substack{n_\ell\mid\gamma+4k+i\\n_{\ell}\notin\mathcal{F}_1(k)}}1
\le 
{\sum}''\left(\left\lfloor\frac{B_k}{n_\ell}\right\rfloor+1\right)\\
&\le\pi(2A_kB_k)+B_k\sum_{n_\ell\notin\mathcal{F}_1(k)}\frac{1}{n_\ell}
\le \frac{\delta_k}{32k}B_k+\frac{1}{2}B_k\leq B_k,
\end{align*}
where the sum ${\sum}''$ is taken over all integers $n_\ell$ with 
$n_\ell\notin \mathcal{F}_1(k)$ and $n_\ell<2A_kB_k$.
Thus, we have $N_{i}\leq 2B_k\xi^{-|i|}$ for each $i$, and hence the number of $\gamma\in\mathcal{G}(k)$
such that (\ref{eq:12}) holds  for some $i$ is at most   
\begin{align*}
\sum_{-\gamma-4k<i\leq -4k}N_{i}+\sum_{4k-3<i\leq A_kB_k/2}^{}N_{i}
&\le4B_k\sum_{i=4k-3}^{\infty}\xi^{-i}=\frac{4\xi^3}{1-\xi^{-1}}B_k\xi^{-4k}\\
&<\frac{B_k}{4}\cdot\exp\left(-16k\sum_{\ell>\mu_k}\frac{1}{n_\ell}\right)=
\frac{B_k}{4}\delta_k.
\end{align*}
Combining \eqref{eq:097} and the above result shows that 
there exist at least $(1-\delta_k/4)B_k$ integers $\gamma\in\mathcal{G}(k)$ such that 
\begin{equation}\label{eq:9382}
\lambda_2(\gamma+4k+i)< \xi^{|i|}/2
\end{equation}
holds for any integer $i$ with  $-\gamma-4k<i\leq -4k$ or $i\geq 4k-3$. 
Therefore, Lemma~\ref{lem:3} follows by \eqref{eq:15}, \eqref{eq:11}, and \eqref{eq:9382}. 
\end{proof}
\subsection{The case where one of $\mathcal{E}_1$ and $\mathcal{E}_3$ is finite}\label{subsec2.2}
Since one of $\mathcal{E}_1$ and $\mathcal{E}_3$ is finite, 
as mentioned at the beginning of Section~\ref{sec:2}, 
we may assume that $n_\ell\equiv1$ (mod $4$) for every $\ell\ge1$ 
or $n_{\ell}\equiv 3$ (mod $4$) for every $\ell\geq1$. 
In any case,  the set of the sequences $b_j(n)$ $(j=1,2,3,4)$ in \eqref{bs} 
coincides with the set of the sequences   
\begin{equation}\label{cs}
c_j(n):=\left\{
\begin{array}{cl}
a(n)&n\equiv j\pmod{4},\\
0&\text{otherwise}
\end{array}
\right.
\end{equation}
for $j=1,2,3,4$, where $a(n)$ denotes the number of divisors $n_\ell$ of $n$. 
Then, in the same way as in $\S~\ref{subsec2.1}$, we consider the system of the $8k-3$ simultaneous congruences 
\eqref{mod4} and \eqref{cong1} with $u_i:=n_i$ $(i\geq1)$ 
(the integers $v_n$ are not used in the congruences), and find that 
there exists a unique integer solution $\eta_k$
of this system of the simultaneous congruences, where 
$0\le\eta_k<A_k:=
4u_1u_2\cdots u_{x_{8k-4}}$.
Using these integers $A_k$ and $\eta_k$,
we define the numbers  $\mu_k:=x_{8k-4}$, $\delta_k$, $\nu_k$, $B_k$ and $\mathcal{G}(k)$
in exactly the same way as in Subsection~\ref{subsec2.1}. 
Under this situation, similarly to the proof of Lemmas~\ref{lem:1} and \ref{lem:3}, we can obtain the following lemmas.
\begin{lemma}\label{lem:2}
Let $c_j(n)$ $(j=1,2,3,4)$ be the sequences defined in \eqref{cs}. 
For every $m=1,2,\dots,8k-4$, let $q$ and $r$ be defined by \eqref{8qr}. 
Then the set $\mathcal{G}(k)$ contains at least $\delta_kB_k/2$
integers $\gamma$ such that 
\begin{align}\label{b_values2}
c_j(\gamma+m)&=
\left\{
\begin{array}{rl}
0&\quad \text{\rm if}\,\,\,r\not\equiv j\pmod{4},\\
k^r2^{k-1}&\quad \text{\rm if}\,\,\, q=k-1,\,\,\,r\equiv j\pmod{4},\\
2^{q}&\quad \text{\rm otherwise}
\end{array}
\right.
\end{align}
for any $m=8q+r=1,2,\dots,8k-4$ and for any $j=1,2,3,4$. 
\end{lemma}
\begin{lemma}\label{lem:4} 
Let $\xi>1$ be an arbitrary constant and assume that $k$ is sufficiently large 
depending on $\xi$.
Then there exist at least $(1-\delta_k/4)B_k$ integers $\gamma\in \mathcal{G}(k)$ such that 
\begin{equation}\label{eq:1034}
c_j(\gamma+4k+i)< \xi^{|i|}
\end{equation}
for any integer $j=1,2,3,4$ and for any integer $i$ with 
$-\gamma-4k<i\leq -4k$ or $i\geq4k-3$. 
\end{lemma}
\section{Proof of Theorem~\ref{thm:1}}\label{Sec3}
In this section,  let $A_k$, $\eta_k$, $\mu_k$, $\delta_k$, $\nu_k$, $B_k$, and $\mathcal{G}(k)$ be as in $\S$~\ref{subsec2.1} or $\S$~\ref{subsec2.2}. 
By Lemmas~\ref{lem:1} (resp. \ref{lem:2}) and \ref{lem:3} (resp. \ref{lem:4}), the number of integers 
$\gamma\in\mathcal{G}(k)$ satisfying 
\eqref{b_values1} (resp. \eqref{b_values2}) and \eqref{eq:103} (resp. \eqref{eq:1034}) is at least 
\[
\frac{\delta_k}{2}B_k+\left(1-\frac{\delta_k}{4}\right)B_k-B_k=\frac{\delta_k}{4}B_k
=\frac{\delta_k}{4A_k}\cdot{A_kB_k}
\ge
\frac{\delta_k}{4A_k}\cdot\frac{128kA_k}{\delta_k}=32k,
\]
where we used (\ref{AB_bound}). 
Thus, we obtain 
\begin{prop}\label{prop:1}
If $\mathcal{E}_1$ and $\mathcal{E}_2$ are both infinite, then 
there exists an integer $\gamma_0\in\mathcal{G}(k)$ such that
the properties \eqref{b_values1} 
and \eqref{eq:103} are fulfilled. 
Similarly, if one of $\mathcal{E}_1$ and $\mathcal{E}_2$ is finite, 
there exists an integer $\gamma_0\in\mathcal{G}(k)$ such that
the properties \eqref{b_values2} 
and \eqref{eq:1034} are fulfilled. 
\end{prop}
Now we prove Theorem~\ref{thm:1}. 
\begin{proof}[Proof of Theorem~\ref{thm:1}]
We first show Theorem~\ref{thm:1} in the case where 
$\mathcal{E}_1$ and $\mathcal{E}_2$ are both infinite.  
Fix $\gamma_0\in \mathcal{G}(k)$ in Proposition~\ref{prop:1}. 
Define $\varepsilon_j:=0$ if $j=1$ or $3$ and $\varepsilon_j:=1$ 
if $j=2$ or $4$. Since $\gamma_0\equiv0\pmod{4}$, 
by (\ref{b_values1}) 
\begin{align*}
\sum_{n=\gamma_0+1}^{\gamma_0+8k-4}\frac{b_j(n)}{\alpha^n}
&=\sum_{m=1}^{8k-4}\frac{b_j(\gamma_0+m)}{\alpha^{\gamma_0+m}}\\
&=\sum_{r=1}^{8}\sum_{q=0}^{k-2}\frac{b_j(\gamma_0+8q+r)}{\alpha^{\gamma_0+8q+r}}
+\sum_{r=1}^{4}\frac{b_j(\gamma_0+8(k-1)+r)}{\alpha^{\gamma_0+8(k-1)+r}}\\
&=\sum_{h=0}^{3}\sum_{q=0}^{k-2}\frac{2^q}{\alpha^{\gamma_0+8q+2h+1+\varepsilon_j}}
+\frac{k^j2^{k-1}}{\alpha^{\gamma_0+8(k-1)+j}}
+\frac{2^{k-1}}{\alpha^{\gamma_0+8(k-1)+4-j+2\varepsilon_j}}\\
&=\frac{(\alpha^2+1)(\alpha^4+1)}{\alpha^{\gamma_0-1+\varepsilon_j}(\alpha^8-2)}
\left(1-\left(\frac{2}{\alpha^8}\right)^{k-1}\right)
+\frac{k^j\alpha^{4-j}+\alpha^{j-2\varepsilon_j}}{\alpha^{\gamma_0+8(k-1)+4}}2^{k-1}.
\end{align*}
Hence, we obtain by \eqref{fs} 
\begin{align*}
P_{k}^{(j)}
&:=
\alpha^{\gamma_0}
\left(
f_j(\alpha^{-1})
-\sum_{n=1}^{\gamma_0}\frac{b_j(n)}{\alpha^n}
-\frac{(\alpha^2+1)(\alpha^4+1)}{\alpha^{\gamma_0-1+\varepsilon_j}(\alpha^8-2)}
\right)\\
&=
\left(
\frac{k^j\alpha^{4-j}+\alpha^{j-2\varepsilon_j}}{\alpha^4}
-\frac{(\alpha^2+1)(\alpha^4+1)}{\alpha^{\varepsilon_j-1}(\alpha^8-2)}
\right)
\left(\frac{2}{\alpha^8}\right)^{k-1}
+
\alpha^{\gamma_0}\sum_{n=\gamma_0+8k-3}^{\infty}\frac{b_j(n)}{\alpha^n}
\end{align*}
for $j=1,2,3,4$. 
As mentioned at the beginning of Section~\ref{sec:2}, we find that 
$|\alpha|^8>|\alpha|^5>2$. Hence, choosing $\xi$ with $1<\xi<\sqrt[8]{2}$ in 
(\ref{eq:103}), we have  
\begin{align*}
\left|
\alpha^{\gamma_0}\sum_{n=\gamma_0+8k-3}^{\infty}
\frac{b_j(n)}{\alpha^n}
\right|
&\leq \frac{1}{|\alpha|^{4k}}
\sum_{i=4k-3}^{\infty}\frac{|b_j(\gamma_0+4k+i)|}{|\alpha|^{i}}\\
&<\frac{1}{|\alpha|^{4k}}
\sum_{i=4k-3}^{\infty}
\left(
\frac{\sqrt[8]{2}}{|\alpha|}
\right)
^i
=O\left(
\left(\frac{2}{|\alpha|^8}
\right)^{k}\right)
\end{align*}
as $k\to\infty$, and therefore, 
\begin{equation}\label{P_j2}
P_{k}^{(j)}
=
\left(
\frac{k^j\alpha^{4-j}+\alpha^{j-2\varepsilon_j}}{\alpha^4}
-\frac{(\alpha^2+1)(\alpha^4+1)}{\alpha^{\varepsilon_j-1}(\alpha^8-2)}
\right)
\left(\frac{2}{\alpha^8}\right)^{k-1}
+
O\left(
\left(\frac{2}{|\alpha|^8}
\right)^{k}\right)
\end{equation}
as $k\to\infty$ for $j=1,2,3,4$. 

Let $\alpha$ be an algebraic integer given in Theorem~\ref{thm:1} 
and let $\alpha_1,\alpha_2,\dots,\alpha_m$ $(|\alpha_i|<1)$ 
be the conjugates of $\alpha$ over $\mathbb{Q}$ other than itself and its complex conjugate. 
Now we choose and fix a constant $\xi$ with $1<\xi<\sqrt[8]{2}$
satisfying $2\xi^{6m}<|\alpha|^8$ and $\xi|\alpha_i|<1$ for any $i=1,2,\dots,m$. 
Suppose to the contrary that the numbers \eqref{5numbers} are linearly dependent over $\mathbb{Q}(\alpha)$, namely, there exist  algebraic integers $\rho_j\in\mathbb{Q}(\alpha)$ $(j=1,2,3,4)$, not all zero, such that 
\[
\Theta:=\rho_1f_1(\alpha^{-1})+\rho_2 f_2(\alpha^{-1})+
\rho_3f_3(\alpha^{-1})+\rho_4 f_4(\alpha^{-1})\]
belongs to the field $\mathbb{Q}(\alpha)$. 
Let $\sigma_i:\mathbb{Q}(\alpha)\rightarrow\mathbb{C}$ be
the $m$ embeddings with $\sigma_i(\alpha)=\alpha_i$ $(i=1,2,\dots,m)$, and 
$d$ be a positive integer such that both 
$d\Theta$ and $d(\alpha^8-2)^{-1}$ are algebraic integer. 
Define 
\begin{align}
\Theta_k
&:=\rho_1P_{k}^{(1)}+\rho_2P_{k}^{(2)}+\rho_3P_{k}^{(3)}+\rho_4P_{k}^{(4)}\label{eq:753}\\
&=\alpha^{\gamma_0}
\Theta
-\alpha^{\gamma_0}
\sum_{j=1}^{4}\sum_{n=1}^{\gamma_0}\frac{\rho_jb_j(n)}{\alpha^n}
-\frac{(\alpha^2+1)(\alpha^4+1)}{\alpha^8-2}
\sum_{j=1}^{4}\rho_j\alpha^{1-\varepsilon_j}. \nonumber
\end{align}
Let $\overline{\Theta}_k$ be a complex conjugate of $\Theta_k$. 
Since  $d\Theta_k$ is an algebraic integer in $\mathbb{Q}(\alpha)$, the norm of $d\Theta_k$ over $\mathbb{Q}$
\begin{equation}\label{eq:236}
\mathcal{N}_k:={{N}_{\mathbb{Q}(\alpha)/\mathbb{Q}}}(d\Theta_k)=d^{m+1}\Theta_k\Phi_k\prod_{i=1}^{m}\Theta_k^{\sigma_i}
\end{equation}
is a rational integer, 
where $\Phi_k:=d\overline{\Theta}_k$, if $\Theta_k\neq \overline{\Theta}_k$, $:=1$, otherwise. 
By (\ref{eq:753}) we have  
\begin{equation}\label{eq:11061}
\Theta_k^{\sigma_i}
=\alpha_i^{\gamma_0}
\Theta^{\sigma_i}
-\alpha_{i}^{\gamma_0}
\sum_{j=1}^{4}\sum_{n=1}^{\gamma_0}\frac{\rho_j^{\sigma_i}b_j(n)}{\alpha_i^n}
-\frac{(\alpha_i^2+1)(\alpha_i^4+1)}{\alpha_i^8-2}
\sum_{j=1}^{4}\rho_j^{\sigma_i}\alpha_i^{1-\varepsilon_j}
\end{equation}
for $i=1,2,\dots,m$. 
Let $\rho:=\max\{|\rho_j^{\sigma_i}|\mid i=1,2\dots,m,j=1,2,3,4\}$.
Using the property (\ref{eq:103}), we obtain  
\begin{align*}
\left|\alpha_{i}^{\gamma_0}
\sum_{j=1}^{4}\sum_{n=1}^{\gamma_0}\frac{\rho_j^{\sigma_i}b_j(n)}{\alpha_i^n}\right|
&=
\left|\sum_{j=1}^{4}
\sum_{n=4k}^{\gamma_0+4k-1}
\frac{\rho_j^{\sigma_i}
b_j(\gamma_0+4k-n)}{\alpha_i^{4k-n}}\right|
\le4\rho\sum_{n=4k}^{\infty}\xi^{n}|\alpha_i|^{n-4k}\\
{}&\le4\rho\xi^{4k}\sum_{n=0}^{\infty}(\xi|\alpha_i|)^{n}
\le\xi^{5k}
\end{align*}
for large $k$, so that by (\ref{eq:11061})    
\begin{equation}\label{eq:11062}
|\Theta_k^{\sigma_i}|\leq 
|\alpha_i|^{\gamma_0}|\Theta^{\sigma_i}|+
\xi^{5k}+
\xi^k
<
|\Theta^{\sigma_i}|+
\xi^{5k}+
\xi^k
\leq 
\xi^{6k}
\end{equation}
for any $i=1,2,\dots,m$. 
Moreover, we have by \eqref{P_j2}
\[
\Theta_k=
D_k\left(\frac{2}{\alpha^8}\right)^{k-1}
+
O\left(
\left(\frac{2}{|\alpha|^8}
\right)^k
\right),
\]
where 
\[
D_k:=\sum_{j=1}^{4}\rho_j\left(
\frac{k^j\alpha^{4-j}+\alpha^{j-2\varepsilon_j}}{\alpha^4}
-\frac{(\alpha^2+1)(\alpha^4+1)}{\alpha^{\varepsilon_j-1}(\alpha^8-2)}
\right).
\]
Note that $|D_k|\rightarrow\infty$ $(k\rightarrow\infty)$, since the 
$\rho_j$ are not all zero. 
Hence, the number $\Theta_k$ does not vanish and 
\begin{equation}\label{eq:305}
0<|\Theta_k|<k^5\left(\frac{2}{|\alpha|^8}
\right)^{k}
\end{equation}
for sufficiently large $k$.  
Since $|\overline{\Theta}_k|=|\Theta_k|\rightarrow0$ $(k\rightarrow\infty)$, 
we have $|\Phi_k|\leq1$, and therefore by (\ref{eq:236}), (\ref{eq:11062}), and (\ref{eq:305})
\[1\leq 
|\mathcal{N}_k|\leq d^{m+1}k^5
\left(
\frac{2}{|\alpha|^8}
\right)^{k}
\xi^{6mk}\leq k^6
\left(
\frac{2\xi^{6m}}{|\alpha|^8}
\right)^k.
\]
This is a contradiction, since $2\xi^{6m}<|\alpha|^8$ 
by our choice of $\xi$. 
Thus, the proof of Theorem~\ref{thm:1} is completed 
in the case where $\mathcal{E}_1$ and $\mathcal{E}_3$ are both infinite. 

Next we consider the case where one of $\mathcal{E}_1$ and $\mathcal{E}_3$ is finite. 
As mentioned at the beginning of subsection~\ref{subsec2.2}, 
the set of functions $f_j(z)$ $(j=1,2,3,4)$ coincides with the sets of the functions 
\begin{equation}\label{gs}
h_j(z):=\sum_{j=1}^{\infty}c_j(n)z^n,\qquad j=1,2,3,4,
\end{equation}
where the sequences $c_j(n)$ $(j=1,2,3,4)$ are defined in \eqref{cs}. 
Similarly as in the previous case, we find by Proposition~\ref{prop:1} that 
there exists an integer $\gamma_0\in\mathcal{G}(k)$ such that 
\begin{equation}\label{c}
\sum_{n=\gamma_0+1}^{\gamma_0+8k-4}\frac{c_j(n)}{\alpha^n}=
\frac{\alpha^4+1}{\alpha^{\gamma_0-j-4}(\alpha^8-2)}
\left(
1-\left(
\frac{2}{\alpha^8}
\right)^{k-1}
\right)
+
\frac{k^j}{\alpha^{\gamma_0+8(k-1)+j}}2^{k-1}
\end{equation}
for $j=1,2,3,4$. Hence, we have by  \eqref{gs} and \eqref{c}  
\begin{align*}
Q_{k}^{(j)}
&:=
\alpha^{\gamma_0}
\left(
h_j(\alpha^{-1})
-\sum_{n=1}^{\gamma_0}\frac{c_j(n)}{\alpha^n}
-\frac{\alpha^4+1}{\alpha^{\gamma_0-j-4}(\alpha^8-2)}
\right)\\
&=
\left(
\frac{k^j}{\alpha^j}
-\frac{\alpha^{j+4}(\alpha^4+1)}{\alpha^8-2}
\right)
\left(\frac{2}{\alpha^8}\right)^{k-1}
+
\alpha^{\gamma_0}\sum_{n=\gamma_0+8k-3}^{\infty}\frac{c_j(n)}{\alpha^n}
\end{align*}
for $j=1,2,3,4$. 
The rest of the proof is completely the same as in 
the case where $\mathcal{E}_1$ and $\mathcal{E}_3$ 
are both infinite. 
\end{proof}

\section*{Acknowledgements} 
The authors would like to deeply thank 
Professor Hajime Kaneko for pointing out to us the reference~\cite{Gar}. 
They also express their sincere gratitude to Professor Joseph Vandehey
for his comments on Erd\H os's paper \cite{Erd2}. 
This work was supported by JSPS KAKENHI Grant Numbers
JP16J00906 and JP18K03201.



\end{document}